\documentclass[11pt,reqno]{amsart}

\usepackage[margin=1in]{geometry}
\usepackage{amsmath,amssymb,amsthm,mathtools}
\usepackage{xcolor}
\usepackage{hyperref}
\usepackage{tikz-cd}
\newtheorem{theorem}{Theorem}[section]
\newtheorem{proposition}[theorem]{Proposition}
\newtheorem{lemma}[theorem]{Lemma}
\newtheorem{definition}[theorem]{Definition}
\newtheorem{corollary}[theorem]{Corollary}
\theoremstyle{definition}
\newtheorem{remark}[theorem]{Remark}

\newcommand{\ii}{\sqrt{-1}}
\newcommand{\calO}{\mathcal O}

\newcommand{\C}{\mathbb C}
\newcommand{\p}{\partial}
\newcommand{\bp}{\overline\partial}
\newcommand{\bL}{\mathbb L}

\newcommand{\dbar}{\overline\partial}

\DeclareMathOperator{\I}{\text{Id}}
\renewcommand{\leq}{\leqslant}
\renewcommand{\geq}{\geqslant}

\renewcommand{\epsilon}{\varepsilon}
\renewcommand{\bar}{\overline}

\numberwithin{equation}{section}

\title{Compact K\"ahler surfaces with semipositive anticanonical bundle}

\begin{document}
\author{Yifan Chen}
\address{Universit\`a di Roma Tor Vergata, Dipartimento di Matematica, Via della Ricerca Scientifica 1, 00133 Roma, Italy}
\email{chen@mat.uniroma2.it}
\author{Simion Filip}
\address{Department of Mathematics, University of Chicago, 5734 S University Ave, Chicago, IL 60637}
\email{sfilip@math.uchicago.edu}
\author{Song Sun}
\address{Institute for Advanced Study in Mathematics, Zhejiang University, Hangzhou 310058, China}
\email{songsun@zju.edu.cn}
\author{Valentino Tosatti}
\address{Courant Institute School of Mathematics, Computing, and Data Science, New York University, 251 Mercer St, New York, NY 10012}
\email{tosatti@cims.nyu.edu}
\author{Junsheng Zhang}
\address{Courant Institute School of Mathematics, Computing, and Data Science, New York University, 251 Mercer St,  New York, NY 10012}
\email{jz7561@nyu.edu}

\begin{abstract}
We classify compact K\"ahler surfaces with semipositive anticanonical bundle. In particular, we resolve a question of Demailly--Peternell--Schneider: any blowup of $\mathbb{P}^2$ at $9$ points on a smooth cubic has semipositive anticanonical bundle.
\end{abstract}
\maketitle

\section{Introduction}
\subsection{Embeddings of curves in surfaces with degree zero normal bundle}
Let $M$ be a complex surface and $D$ be a smooth compact connected divisor. Denote by $L$ the holomorphic normal bundle of $D$.
We assume that $L$ is numerically trivial, i.e., $\deg L=D^2=0.$
The line bundle $\mathcal{O}_M(D)$ on $M$ restricts to $L$ along $D$, and this restriction admits a flat Hermitian metric. A natural question one can ask is whether this flatness can be propagated to a neighborhood of $D$: does there exists an open neighborhood $D\subset U\subset M$ such that $\mathcal{O}_M(D)|_U$ admits a flat Hermitian metric?


This question was first studied in the foundational work of Ueda \cite{Ueda} (see also Arnold \cite{arnold}) for a general smooth compact connected curve embedded in a complex surface with degree zero normal bundle. In this setting, he was able to show that the answer is affirmative if $L$ is torsion or satisfies a Diophantine condition. Further work of Brunella \cite{brunella} and Koike \cite{Koike2023} also studied this problem, and Ogawa \cite{ogawa} relaxed the Diophantine condition to the weaker arithmetic ``Brjuno condition''. Our first result is a resolution of this question for anticanonical curves, which by adjunction have genus $1$:

\begin{theorem}\label{thm:dps-semipositivity}
Let $M$ be a complex surface and $D$ be a smooth compact connected divisor $D\in|K^{-1}_M|$. Then there exists an open neighborhood $U$ of $D$, such that $\mathcal{O}_M(D)|_U$ admits a flat Hermitian metric. In particular,  $K_M^{-1}\simeq\mathcal{O}_M(D)$ is Hermitian semipositive, i.e. admitting a smooth Hermitian metric $h$ with curvature $R_h\geq 0$ on $M$.
\end{theorem}

Note that no Diophantine or arithmetic condition is needed, and $M$ is not assumed to be compact. The fundamental reason for this comes from the assumption that $D\in |K^{-1}_M|$, as we will explain below. Without this assumption, embeddings of an elliptic curve into a noncompact surface $M$ with degree zero normal bundle for the conclusion of Theorem \ref{thm:dps-semipositivity} fails were constructed by Ueda \cite[\S5.4]{Ueda}. More recently, Koike-Ogawa \cite{KO} showed that in these examples we even have that $\mathcal{O}_M(D)$ is not Hermitian semipositive.

\subsection{Semipositive Ricci curvature}
The case when the surface $M$ is compact K\"ahler is of particular interest, since by Yau's Theorem, the condition that $K_M^{-1}$ is Hermitian semipositive is equivalent to $M$ admitting a K\"ahler metric with nonnegative Ricci curvature.

The problem of classifying compact K\"ahler surfaces with nonnegative Ricci curvature is also classical, and was famously studied by Yau, who obtained a partial classification of such surfaces \cite[Theorem 3]{yau}. It was also later investigated by Demailly--Peternell--Schneider \cite[Example~2 and p.~223]{DPS1996}, who raised two questions at the end of their paper. The first question, which asked whether every rationally connected compact
K\"ahler manifold with nef anticanonical bundle has Hermitian semipositive
anticanonical bundle, was answered negatively by Koike \cite{Koike2017}, who showed that there exist counterexamples which are obtained from $\mathbb{P}^2$ by blowing up $9$ points along a nodal cubic. 

The second question they posed remained open: does every blow up of $\mathbb{P}^2$ along $9$ points on a {\em smooth} cubic have Hermitian semipositive
anticanonical bundle? The $9$ points are allowed to be ``infinitely near'' in the sense that, after each blow-up, the center of the next blow-up may be any point
on the strict transform of the cubic. This question was subsequently studied by a number of people, including Sad \cite{Sad}, Eckl \cite[\S 4.3]{eckl},  Brunella \cite{brunella}, Claudon-Loray-Pereira-Touzet \cite{CLPT}, and Koike \cite{Koike2017, Koike2023}.
Theorem~\ref{thm:dps-semipositivity} immediately answers this question affirmatively:

\begin{corollary}
Let $M$ be the blowup of $\mathbb{P}^2$ at $9$ points (possibly infinitely near) on a smooth cubic. Then $K_M^{-1}$ is Hermitian semipositive.
\end{corollary}

Furthermore, combining Theorem \ref{thm:dps-semipositivity} with further results by Koike  \cite{Koike2023,Koike2026} and \cite{Sakai2001}, we obtain a complate classification of compact K\"ahler surfaces with semipositive anticanonical bundle, or equivalently of compact K\"ahler surfaces that admit a K\"ahler metric with non-negative Ricci curvature:

\begin{theorem}
\label{thm:surface-semipositivity-classification}
Let $M$ be a smooth compact K\"ahler surface.  Then $K^{-1}_M$ is
Hermitian semipositive if and only if exactly one of the following holds.
\begin{enumerate}
\item[\textup{(A)}] The line bundle $K^{-1}_M$ is semiample.
\item[\textup{(B)}] The surface $M$ is rational and $\lvert K^{-1}_M\rvert$ contains a smooth elliptic curve $D$ such that $N_{D} \in \operatorname{Pic}^0(D) $ is non-torsion.
\item[\textup{(C)}] The surface $M$ is rational and $|K^{-1}_M|$ contains a
reduced cycle $D$ of rational curves such that $N_D:=\calO_D(D)\in\operatorname{Pic}^0(D)\simeq\C^*$ whose period $q_D$ satisfies that $|q_D|=1$ but is not a root of unity.
\item[\textup{(D)}] There are an elliptic curve $B$ and a non-torsion
line bundle $L\in\operatorname{Pic}^0(B)$ such that
\[
       M\simeq\mathbb P_B(\calO_B\oplus L).
\]
\end{enumerate}
\end{theorem}

Here by \emph{reduced cycle of rational curves} we mean a connected nodal curve $ D=D_1+\cdots+D_r$ in which every irreducible component appears with multiplicity one and has normalization $\mathbb P^1$, and whose dual graph is homeomorphic to a circle. We allow the one-component nodal cycle and the two-component cycle whose components meet transversely at two points. 
For such a cycle $D$, choose an algebraic group isomorphism
\[
        \alpha_D:\operatorname{Pic}^0(D)\longrightarrow\C^*
\]
and, whenever $N_D\in\operatorname{Pic}^0(D)$, set
$q_D:=\alpha_D(N_D)$.  Any other choice replaces $q_D$ by $q_D^{-1}$, so the conditions $|q_D|=1$, torsion, and the torsion order are independent of the choice.

\begin{remark}
It is well-known that in case (A) of Theorem \ref{thm:surface-semipositivity-classification}, we can further refine the classification based on the value of the Iitaka dimension of $K_M^{-1}$, which is either $0,1,$ or $2$. When $\kappa(K_M^{-1})=2,$ $M$ is a weak del Pezzo surface, which is isomorphic to $\mathbb{P}^2,$ $\mathbb{P}^1\times\mathbb{P}^1,$ the Hirzebruch surface $\mathbb{F}_2$, or a blowup of $\mathbb{P}^2$ in up to $9$ points in ``almost general position'', see e.g. \cite[\S 8.1.3]{dolgachev}. When $\kappa(K_M^{-1})=1$, we must have $H^{2,0}(M)=H^0(M,K_M)=0$ so $M$ is projective, and by \cite[Prop. 1.5]{BauerPeternell2004} we have that $M$ is either finitely covered by $\mathbb{P}^1\times E$ with $E$ elliptic, or it is the blowup of $\mathbb{P}^2$ at the $9$ base points of a Halphen pencil (of some index $m\geq 1$). And when $\kappa(K_M^{-1})=0$, we have $K_M$ holomorphically torsion, so $M$ is $K3,$ a torus, an Enriques surface, or a bielliptic surface.
\end{remark}

\begin{remark}
\label{rem:non-kahler-surfaces}
For a compact non-K\"ahler surface $M$, the anticanonical bundle $K^{-1}_M$ is
Hermitian semipositive if and only if $M$ is a primary or secondary Kodaira
surface, or a Hopf surface that is finitely covered by a diagonal primary Hopf
surface. This is well known, but we were unable to locate a proof in
the literature, so we include one in Section \ref{sec:proof of remark} for completeness.
\end{remark}
\subsection{Applications}
\subsubsection{Ueda Theory}
We give an application of this result to Ueda theory \cite{Ueda, Neeman}, see Section \ref{ue} below for the basic definitions. An embedding $D\subset M$ of a connected compact complex curve into a complex surface, with degree zero normal bundle, is called of Ueda type $(\gamma)$ if the embedding is formally vertically linearizable (equivalently, all its Ueda classes vanish), but not vertically linearizable (which, when $M$ is compact, is equivalent to $\mathcal{O}_M(D)$ not being Hermitian semipositive). No such embeddings exist if $D$ has genus zero \cite{savelev}, while type $(\gamma)$ embeddings of an elliptic curve into a noncompact complex surface were constructed by Ueda \cite[\S5.4]{Ueda} by taking the suspension of a germ of a local biholomorphism defined near the origin in $\mathbb{C}$ with a Cremer fixed point at the origin, see also \cite{KO}. 

Constructing an example of a type $(\gamma)$ embedding into a \textit{compact} complex surface is a well-known open problem \cite[p.1464]{CLPT}, \cite[p.242]{Koike2023}, \cite[p.24]{Koike2024}. As a corollary of Theorem \ref{thm:surface-semipositivity-classification}, we show that no such embeddings exist when $D$ has genus $1$:

\begin{theorem}\label{ueda}
There is no embedding $D\subset M$ of a smooth elliptic curve $D$ into a compact complex surface $M$ with $\deg N_{D/M}=0$ and which is of Ueda type $(\gamma).$
\end{theorem}

\subsubsection{Stein Complements}
Our last application of our results is to another well-studied question in the literature: given a compact complex surface $S$ containing a smooth curve $C$ with $(C^2)=0$, find necessary and sufficient conditions for $S\backslash C$ to be Stein. An obvious necessary condition is that $S\backslash C$ does not contain any compact curve or any compact Levi-flat real hypersurface, and an obvious sufficient condition is that $S\backslash C$ be the underlying analytic space of an affine algebraic variety. See for example \cite[Problem VI.3.4]{hartshorne}, \cite{Neeman2}, \cite{KP}, \cite{HP} and \cite{FX} for more on this, including examples of Stein but non-affine complements.

\begin{corollary}\label{stein}
Let $S$ be a compact complex surface and $C\subset S$ a smooth connected anticanonical divisor with $C^2=0.$ Then $S\backslash C$ is not Stein.
\end{corollary}

The special case when $S$ is the blowup of $\mathbb{P}^2$ along $9$ (possibly infinitely near) points that lie on a smooth cubic and $C$ is the strict transform of that cubic was considered explicitly by Ogus \cite[p.223]{Ogus}, and further studied by Neeman \cite{Neeman}, and others. Brunella \cite[Remark 4]{brunella} had conjectured Corollary \ref{stein} in this case.

\subsection{Sketch of the proof}
We explain here the main idea of the proof of Theorem~\ref{thm:dps-semipositivity}. 
Under the assumptions of Theorem~\ref{thm:dps-semipositivity}, it suffices to construct a pluriharmonic function on the end of $X:=M\setminus D$ with logarithmic growth near $D$. Such a pluriharmonic function $z$ exists on the total space of punctured normal bundle $L^\times: = L\setminus {\bf 0}_L$ of $D$ with respect to the complex structure $J_L$. Geometrically, it is the $\mathbb R$ parameter of a cylindrical K\"ahler Ricci flat metric on $L^\times$. Using an exponential diffeomorphism $\Phi$, one can identify a neighborhood of $D$ in $M$ with a neighborhood of the zero section ${\bf 0}_L$ in $L$, and it is then natural to consider the pullback $\Phi^*z$. One would then like to find a function $u$ with growth rate slower than logarithmic, satisfying
\begin{equation}
    \ii\partial_{J_M}\dbar_{J_M}u = \ii\partial_{J_M}\dbar_{J_M}\Phi^*z.
\end{equation}
Although this is a linear equation, there appears to be no direct method for solving it with the estimates required here.
Therefore we replace this linear problem for $u$ by a nonlinear problem for the diffeomorphism $\Phi$, that is to say, we want to find a diffeomorphism $\Psi$ such that 
\begin{equation}\label{eq-ph in introduction}
\ii\partial_{J_M}\dbar_{J_M}\Psi^*z= \Psi^*(\ii\partial_{\Psi_*J_M}\dbar_{\Psi_*J_M}z) =0
\end{equation} on an end of $X$. 
A key ingredient in this construction is the closed $J_M$-holomorphic volume form $\Omega_M$ on $X$, which extends
meromorphically across \(D\) with a simple pole. It allows us to reduce the equation for the diffeomorphism to a first-order equation and, at the same time, to define a bounded operator with quantitative control. Using this operator, we formulate the appropriate nonlinear functional and then apply the Banach fixed-point theorem to find a zero of this functional, which in turn gives a solution to \eqref{eq-ph in introduction}.

As discussed in Section \ref{ue} below, given an embedding of an elliptic curve $D$ into a compact complex surface $M$, with degree zero normal bundle, work of Ueda \cite{Ueda}, Brunella \cite{brunella} and Koike \cite{koike} shows that proving that $\mathcal{O}_M(D)$ is Hermitian semipositive is equivalent to the germ of the embedding $D\subset M$ being ``vertically linearizable''. This is a condition of dynamical nature, which is implied by the stronger condition of ``linearizability'' (i.e. $D$ admitting a holomorphic tubular neighborhood in $M$). While in general linearizability is obstructed, even at the formal level, Ueda \cite{Ueda} observed that all degree-zero embeddings of an elliptic curve into a complex surface are always formally vertically linearizable, and to establish vertical linarizability one has to prove that these formal power series are convergent. Unfortunately, to prove this one runs into the well-known issue of ``small divisors'', which can destroy the convergence of the formal solution, unless one makes arithmetic assumptions on the normal bundle to prevent this. This does indeed happen for the full linearization problem: the second and fourth-named authors have recently shown \cite{FT} that given any smooth cubic in $\mathbb{P}^2$ there are many configurations of $9$ distinct points such that after blowup, the embedding of the strict transform of the cubic in the rational surface is not linearizable and has nontorsion normal bundle.

In the setting of Theorem \ref{thm:dps-semipositivity}, we have the key extra geometric assumption that $D\in |K^{-1}_M|$, which allows us to avoid small divisors, by exploiting  $\Omega_M$.
The holomorphic volume form allows us to turn the desired pluriharmonicity into a first-order
condition on the neighborhood identification. Its closedness relates
the tangential and normal components of the error through
\eqref{eq:closed-horizontal-identities} and \eqref{eq--beta and alpha}, ensuring that the terms affected by small
divisors are themselves proportionally small. The unitary flat
geometry makes this compensation precise, yielding the uniform
estimates \eqref{eq-vertical} and
\eqref{eq:pure-horizontal-Q-estimate}, even when powers of the normal
bundle approach the trivial bundle arbitrarily rapidly. Thus, the
geometric compatibility supplied by the volume form replaces
arithmetic assumptions usually needed to control small divisors.

\subsection{A relation to dynamics}
A closely related technique of proof appears in the work of Perez Marco and Yoccoz \cite{Perez-MarcoYoccoz1994_Germes-de-feuilletages-holomorphes-a-holonomie-prescrite}.
Specifically, the authors consider a germ of holomorphic vector field near the origin in $\mathbb{C}^2$ of the form $\dot{x}=-x(1+\dots)$ and $\dot{y}=\alpha y(1+\dots)$ with $\alpha>0$.
They show that the holonomy map $\mathrm{Hol}_\alpha$ puts such vector fields up to change of coordinates, in bijection with the set of germs of holomorphic diffeomorphisms of the form $z\mapsto e^{2\pi i \alpha}z(1+\dots)$ up to conjugation.
In particular, the classification of either class of objects is equivalent to the classification of the other.
When $\alpha$ satisfies a Diophantine condition, each side of the correspondence can be linearized, and the classification is trivial.
When $\alpha$ is rational, the classification is difficult but possible.
When $\alpha$ is not Diophantine and irrational, the classification is challenging.

More importantly, the proof in \cite{Perez-MarcoYoccoz1994_Germes-de-feuilletages-holomorphes-a-holonomie-prescrite} is independent of any properties of $\alpha$.
Furthermore, the proof goes by constructing first a smooth diffeomorphism, and then showing that the diffeomorphism can be improved to a holomorphic one.
However, the key analytic input are H\"ormander's $L^2$ estimates for the $\dbar$ operator, which are of a different flavor from our Lemma~\ref{l:flat_isometry}.

\subsection*{Acknowledgments}
Y.C. thanks Eleonora Di Nezza for organizing an Oberwolfach workshop ``Algebraic Foliations: Analytic and Birational Viewpoint", for proposing Brunella's paper~\cite{brunella} and this problem, and for fruitful discussions. She also thanks C\'ecile Gachet and Romy Merkel for helpful discussions at Oberwolfach. Y.C. was funded by the European Research Council (ERC) through the SiGMA project (Grant Agreement No. 101125012). V.T. was partially supported by NSF grant DMS-2404599, S.F. was partially supported by NSF grants DMS-2305394 and DMS-2604568.

\subsection*{Declaration on the use of AI}

This method was motivated by a project on the compactification problem for gravitational instantons undertaken by Y.C. together with Chunhui Wei. That project led us to consider a possible approach based on changing the diffeomorphism.
During subsequent discussions with ChatGPT 5.6 Sol, the AI suggested a preliminary version of Lemma \ref{l:flat_isometry}, which we then reformulated and developed into the geometric statement and argument presented here.

\section{Vertical linearization and Ueda Theory}\label{ue}

Let $C\subset S$ be a holomorphic embedding of a smooth connected compact complex curve into a smooth compact complex surface, with $C^2=0$. The normal bundle $N_{C/S}$ is topologically trivial, hence unitary flat. Thus, there exists a finite open cover $\{V_i\}$ of $C$ by coordinate charts in $S$, such that the normal bundle $N_{C/S}$ is trivialized over $\{V_i\cap C\}$ with  constant transition functions $t_{ij}\in\mathbb{C}^*$ with $|t_{ij}|=1$. Write $\Phi_i=(z_i,w_i):V_i\to\mathbb{C}^2$ for the chart maps, which we may assume satisfy $C\cap V_i=\{w_i=0\}$ and $dw_i=t_{ij} dw_j$ on $C\cap V_i\cap V_j$, by the defining property of the normal bundle.

Following Ueda \cite{Ueda}, we say that the embedding $C\subset S$ is {\em vertically linearizable} if we can find such charts with the property that on $V_i\cap V_j$ we have
\begin{equation}\label{vert_lin}
w_i=t_{ij}w_j.
\end{equation}
The reason for this name (which comes from  \cite{GS}) is the following: suppose \eqref{vert_lin} holds, then given any $(z_j,w_j)\in \Phi_j(V_i\cap V_j)$ we have
\begin{equation}
(\Phi_i\circ \Phi_j^{-1})(z_j,w_j)=(z_i(z_j,w_j), t_{ij} w_j),
\end{equation}
which has the same $w$-component as the normal bundle transition maps, but the $z$-component is $z_i(z_j,w_j)$ ($z$-component of the change of charts) instead of the coordinate chart transition map on $C$. Note though that the $z$-component does restrict to the chart transition map along $C$.

We collect here the following useful facts from the literature:
\begin{proposition}\label{ued}
Let $S$ be a compact complex surface and $C\subset S$ a smooth connected compact curve with $\deg N_{C/S}=0.$ Then the following are equivalent:
\begin{enumerate}
\item the line bundle $\mathcal O_S(C)$ admits a smooth Hermitian metric with nonnegative curvature form
\item there is an open neighborhood $C\subset U\subset S$ and a plurisubharmonic function $F:U\to\mathbb{R}\cup\{-\infty\}$ such that $F-\log|w|^2$ is locally bounded for every local defining function $w$ of $C$, and satisfies
\begin{equation}\label{ass}
dd^c F=[C]
\end{equation}
\item there is an open neighborhood $C\subset U\subset S$ such that $\mathcal O_S(C)|_U$ is unitary flat
\item the embedding $C\subset S$ is vertically linearizable
\item  there is an open neighborhood $C\subset U\subset S$ and a nonsingular holomorphic foliation $\mathcal{F}$ on $U$, such that $C$ is a compact leaf, and the holonomy of $\mathcal{F}$ along $C$ is $U(1)$-linear.
\end{enumerate}
\end{proposition}
The compactness of $S$ is used only in the implication (1)$\Rightarrow$(3). It is conjectured in \cite{KO} that this should hold even when $S$ is noncompact.
\begin{proof}

 (1)$\Rightarrow$ (3): This follows from \cite[Corollary 1.5]{koike}, generalizing  earlier work of Brunella \cite{brunella}.

 (3)$\Rightarrow$ (1): This is elementary \cite[p.442]{brunella}, see the end the proof of Theorem \ref{thm:dps-semipositivity} in Section \ref{sek}.

(2)$\Leftrightarrow$ (3): we recall here the simple argument. First assume that (2) holds.  Fix a smooth Hermitian metric $h_0$ on $\mathcal O_S(C)$, a defining section $s$, and combining the Poincar\'e-Lelong formula with our assumption \eqref{ass}, we see that on $U$ we have
\begin{equation}
dd^c(\log|s|^2_{h_0}-F)=-R_{h_0},
\end{equation}
so by elliptic regularity $G:=\log|s|^2_{h_0}-F$ is smooth on $U$. Hence the smooth Hermitian metric $h:=h_0e^{-G}$ on $\mathcal O_S(C)|_U$ has curvature identically zero, and so (3) holds.

Conversely, if (3) holds then there is a smooth Hermitian metric $h$ on $\mathcal O_S(C)|_U$ with curvature identically zero $R_h\equiv 0$. Defining then
\begin{equation}
F:=\log|s|^2_h,
\end{equation}
we obtain a function $F:U\to\mathbb{R}\cup\{-\infty\}$, which is smooth on $U\backslash C$ and comparable to $\log\mathrm{dist}(\cdot,C)$ near $C$, and using the Poincar\'e-Lelong equation, on $U$ we have
\begin{equation}
dd^c F=[C]-R_h=[C],
\end{equation}
as desired.

(3)$\Leftrightarrow$ (4)$\Leftrightarrow$(5): these are due to Ueda \cite[p.596]{Ueda}. To see that (3)$\Rightarrow$(4),  note that on an open cover $\{V_i\}$ as above, the transition functions of $\mathcal O_S(C)$ are given by $\frac{w_i}{w_j}$. If (3) holds, then we can find such a cover such that on $V_i\cap V_j$ we have $\frac{w_i}{w_j}=s_{ij},$ for some $s_{ij}\in\mathbb{C}^*$ with $|s_{ij}|=1$. But on $C\cap V_i\cap V_j$ we have $dw_i=t_{ij}dw_j$, and so $s_{ij}=t_{ij}$ and \eqref{vert_lin} holds. The converse implication (4)$\Rightarrow$(3) is analogous. To see why (4)$\Rightarrow$(5), we assume \eqref{vert_lin}, and observe that the line spanned by the vector field $\frac{\partial}{\partial z_i}$ agrees on $V_i\cap V_j$ with the line spanned by $\frac{\partial}{\partial z_j}$. These thus define a nonsingular holomorphic foliation $\mathcal{F}$ defined on $U$, whose leaves are given locally by $\{w_i={\rm const}\}$, so in particular $C$ itself is a compact leaf. The holonomy of $\mathcal{F}$ along a path contained in $C$ can be viewed as an element of $\mathrm{Bihol}(\mathbb{C},0)$ (after choosing transversal holomorphic discs at the beginning and end points), and all of these elements must be given by multiplication by an element of $U(1)$, because of \eqref{vert_lin}. 
Conversely, assume (5). Choose a base point of $C$ and a transverse
coordinate in which the entire holonomy representation is
$U(1)$-linear. Transport this coordinate along paths in $C$ to
obtain transverse coordinates in a sufficiently fine collection
of flow boxes. On overlaps, the two choices differ by the holonomy
of a loop based at the chosen point, hence by multiplication by a
unit complex number. After shrinking the flow boxes, their
transverse coordinates therefore satisfy
\[
    w_i=s_{ij}w_j,\qquad |s_{ij}|=1,
\]
and arguing as above we see that $s_{ij}=t_{ij}$, and so (4) holds.
\end{proof}

\section{Geometry of cylindrical model space}\label{sec--set-up}
Let $M$ be a complex surface and $D$ be a smooth connected anticanonical elliptic curve. Denote by $J_M$ the complex structure on $M$ and  denote 
$$X:=M\setminus D.$$
Fix a nowhere vanishing holomorphic 1-form $\Omega_D$ on $D$. There is a meromorphic $2$-form $\Omega_M$ on $M$ which is holomorphic on $X$ and with a simple pole along $D$, such that the residue of $\Omega_M$ along $D$ is given by $\Omega_D$. Denote by $$\omega_D=\sqrt{-1}\Omega_D\wedge\bar\Omega_D$$
the flat K\"ahler form. 
Let $v$ be the smooth $(1,0)$ vector field on $D$ defined by $\Omega_D(v)=1.$

Let $L$ be the holomorphic normal bundle of $D$, and denote by $J_L$ the natural complex structure on the total space of $L$.
We have the natural projection map
$$\pi: L\rightarrow D.$$
 Let ${\bf 0}_L$ be the zero section of $L$ and denote
$$L^\times :=L\setminus {\bf 0}_L.$$

Choose a flat Hermitian metric $h$ on $L$ and denote by $\nabla^L$ the induced flat Chern connection on $L$. Then we have a splitting
$$T^{1, 0}L=H\oplus V^{1, 0}L,\quad V^{1,0}L\simeq \pi^*L.$$
Since $\nabla^L$ is flat the splitting is in particular holomorphic, and $H$ can be identified with $\pi^*T^{1, 0}D$. This way, we can view $v$ as a parallel $(1,0)$ vector field on $L$.

Let $\xi$ be the Euler holomorphic vector field on $L$  which generates the natural $\mathbb C^*$ action. Let $\Theta$ be the holomorphic 1-form on $L^\times$ with $\Theta(\xi)=1$ and $\Theta|_H=0$. If we choose a local parallel unit frame $e$ for $L$, then we have fiber holomorphic coordinate $w$ so that $\zeta=w e_{\pi(\zeta)}$ for $\zeta\in L$. Then we can write 
$$\xi=w \p_w, \ \ \ \Theta=w^{-1}dw.$$
Let $\vartheta: = \pi^*\Omega_D$. There is a natural holomorphic 2-form $\Omega_0$ on $L^\times$  given by 
$$\Omega_0=\Theta\wedge \vartheta.$$
 On $L^\times$ we have a natural function pluriharmonic with respect to $J_L$, given by 
$$z:=-\log |\zeta|_h^2,$$
and  a flat cylindrical K\"ahler form and metric
$$\omega=\pi^*\omega_D+ dd^c_{J_L}z^2, \quad g=|\vartheta|^2+dz^2+(J_Ldz)^2,$$ and denote its connection by $\nabla$. 
The corresponding metric $g$ splits isometrically as a product metric on $\mathbb T^3\times \mathbb R$, where $z$ is naturally the coordinate on $\mathbb R$. 

The standard tubular neighborhood theorem allows us to identify a neighborhood of $D$ in $M$ diffeomorphically with a neighborhood $\mathcal N$ of the zero section ${\bf 0}_L$ in $L$, and then we may view $J_M$ as a complex structure on $\mathcal N$. This identification is not unique, and our goal is to find a good identification such that $z$ is pluriharmonic with respect to $J_M$. This is achieved in the next two sections.

\section{The first diffeomorphism}

 For $\rho>0$ we denote by $$\mathcal N_\rho:=\{\zeta\in L: |\zeta|_h\leq \rho\}.$$

\begin{proposition}\label{p:vertical_disc}
	We may smoothly identify a neighborhood of $D$ in $M$ with $\mathcal N_\rho$ for some $\rho>0$, such that for each $p\in D$, $\mathcal N_\rho\cap L_p$ is a $J_M$-holomorphic disk and such that for all $k\geq 0$,
	\begin{equation}\label{e:volume form close}|\nabla^k(\Omega_M-\Omega_0)|_g=O(e^{-\frac{1}{2}z}). \end{equation}
\end{proposition}

\begin{proof}
\noindent \textbf{Step 1}: Fix a smooth Hermitian metric $g_M$ in a neighborhood of $D$. We first apply the natural normal exponential map to get a local diffeomorphism. For any $\zeta\in L\simeq T^{1,0}M|_D/T^{1,0}D$, choose a unique tangent vector representative $V\in T_{\pi(\zeta)}^{1,0}M$ in the orthogonal complement of $T_{\pi(\zeta)}^{1,0}D$ with respect to $g_M$. Define the normal exponential $$\Xi(\zeta)=\exp_{\pi(\zeta),g_M}(2\mathrm{Re} \,V).$$
The restriction $d\Xi|_D:TL|_{\mathbf 0_L}\simeq TD\oplus L\longrightarrow TM|_D$ is given by the above identification. Thus for $|\zeta|_h<\epsilon$, with $\epsilon>0$ sufficiently small, we obtain a diffeomorphism $\Xi:\mathcal N_\epsilon \longrightarrow U\subset M$ onto some open neighborhood $U$ of $D$, and restricting it to $\mathcal{N}_\epsilon\cap L^\times$ we get a diffeomorphism with $U\setminus D$. 
From now on we will naturally identify $U\setminus D$ with $\{z>R_0\}\subset L^\times$ through $\Xi$ and omit $\Xi^*$. 

Such exponential type diffeomorphism gives the estimate \eqref{e:volume form close}, as shown e.g. in \cite[Proposition 3.4]{HSVZ}. One can also prove the closeness of complex structure using the integrability of $J_M$ as in \cite[lemma 2.6]{JZ24}. For the reader's convenience, we provide some details.

We work in local coordinates, let $e$ be a parallel local unitary frame of
$L$, let $w$ be its fiber coordinate, and choose a holomorphic coordinate $z_0$ on the base with $\Omega_D=dz_0$.  Choose holomorphic coordinates $(Z,W)$ on $M$ such
that $D=\{W=0\}$, $Z|_D=z_0$, and $\partial/\partial W|_D$ is identified with $e$.
The residue normalization gives
\[
        \Omega_M=(1+W a(Z,W))\,\frac{dW}{W}\wedge dZ
\]
for a holomorphic function $a$. The map $\Xi$ fixes the zero section, its differential is complex
linear there, and it induces the identity on the normal quotient.  In
the coordinates $(z_0,w)$ this gives
\[
 \begin{aligned}
        Z\circ\Xi&=z_0+b(z_0,\bar z_0)w+O(|w|^2),\\
        W\circ\Xi&=w+O(|w|^2),
 \end{aligned}
\]
where the remainders are smooth with all derivatives.  Put $r:=|w|$.
We write $A=O_{\mathrm{cyl}}(r^\nu)$ if the coefficients of $A$ in the
logarithmic coframe
\[
        dz_0,\quad d\bar z_0,\quad \Theta=\frac{dw}{w},
        \quad \bar\Theta=\frac{d\bar w}{\bar w},
\]
and all their derivatives under
$\partial_{z_0},\partial_{\bar z_0},w\partial_w,$ and
$\bar w\partial_{\bar w}$ are $O(r^\nu)$.  A smooth coefficient which
vanishes to ordinary order $\nu$ in $(w,\bar w)$ is
$O_{\mathrm{cyl}}(r^\nu)$, as follows directly from its Taylor
expansion.

Direct substitution now shows that
\[
        \Omega_M-\Omega_0
        =O_{\mathrm{cyl}}(r).
\]
Indeed, the possible contribution of the linear term $bw$ is
proportional to $\Theta\wedge dw=0$, and every remaining contribution
contains at least one additional factor of $w$ or $\bar w$.  The
Levi--Civita connection of the flat cylindrical metric has bounded
coefficients in our coframe, so the same estimate holds
after every covariant derivative.  Therefore, for every $k\geq0$,
\[
        |\nabla^k(\Omega_M-\Omega_0)|_g
        =O(r)=O(e^{-z/2}).
\]

\noindent \textbf{Step 2:} To make the fibers $L_p$ holomorphic with respect to $J_M$, we need to apply a further implicit function theorem fiberwise. This is a standard practice in the theory of pseudo-holomorphic curves. For the convenience of readers we include the details here. 

 For a vector field $Y$ defined in a neighborhood of ${\bf 0}_L$, we denote $$\Phi_Y(\zeta)=\exp_{\zeta,g_M}(Y_\zeta).$$

For $Y$ sufficiently small in $C^1$, $\Phi_Y$ is a local diffeomorphism in a neighborhood of ${\bf 0}_L$. To preserve \eqref{e:volume form close} we need $Y=O(|\zeta|^2)$ together with the corresponding derivative bounds.

For $p\in D$ and $\rho>0$ small, we denote the disc $$\Delta_{p, \rho}:=L_{p}\cap \mathcal N_\rho$$
and consider 
\begin{equation}\label{CR}\mathcal F_{p, \rho}(Y):=\frac{1}{2}\left(d\Phi_Y|_{T\Delta_{p,\rho}}+J_M\circ d\Phi_Y\circ J_L|_{T\Delta_{p,\rho}}\right).
\end{equation}
Then 
$\mathcal F_{p, \rho}(Y)$ vanishes precisely when $\Delta_{p, \rho}$ is $\Phi_Y^*J_M$ holomorphic. To make precise the implicit function theorem argument, we wish to identify $TL$ with $\pi^*E$ for a vector bundle $E$ on $D$. To do this, we use the Chern connection of $g_M$ on $(T\mathcal N,J_M)$, let
\[
    I_q:T_q\mathcal N\longrightarrow E_p,
    \qquad q\in\Delta_{p,\rho},
\]
denote radial parallel transport, and let
\[
    P_{Y,q}:T_{\Phi_Y(q)}\mathcal N\longrightarrow T_q\mathcal N
\]
denote parallel transport backwards along
$t\mapsto\exp_q(tY_q)$. In \eqref{CR} we then implicitly postcompose with
$I_q\circ P_{Y,q}$.
These complex-linear identifications make the operator a smooth map
into the fixed space of $E_p$-valued $(0,1)$-forms. We suppress them
in the notation.

Let $C_\rho(f)$ be the normalized Cauchy operator on standard disc $\Delta_\rho\subset \mathbb C$, more precisely
$$\mathcal C_\rho: C^{k, \mu}(\overline \Delta_\rho )\rightarrow C^{k+1, \mu}(\overline\Delta_\rho);\quad f\mapsto T_\rho f-T_\rho f(0)-\zeta \cdot \p_{\zeta}T_\rho f(0),$$
with 
$$T_\rho f(\zeta):= \frac{\sqrt{-1}}{2\pi}\int_{\Delta_\rho} \frac{f(w)}{\zeta-w}dwd\bar w. $$
Then we have $\p_{\bar\zeta}\mathcal C_\rho f=f$ and $\mathcal C_\rho f(0)= \p_{\zeta}\mathcal C_\rho f(0)=0$.
Define $\mathcal C_{p,\rho}$ on $C^{k,\mu} (\overline\Delta_{p,\rho}; \Lambda^{0,1}T^*\Delta_{p,\rho} \otimes E_p)$ by 
$$\mathcal C_{p,\rho}(f d\bar\zeta\otimes \eta):=\mathcal C_\rho(f)\otimes \eta.$$ 

Consider the Banach spaces
$$\mathfrak B_{p,\rho}:= \{\beta\in C^{k,\mu} (\overline\Delta_{p,\rho}; \Lambda^{0,1}T^*\Delta_{p,\rho} \otimes E_p) \,|\,\beta(0)=0\}$$ 
with its image
$$ \mathfrak X_{p,\rho}:= \mathcal C_{p,\rho}(\mathfrak B_{p,\rho}) \subset \mathfrak B'_{p,\rho}: = C^{k+1,\mu}(\overline\Delta_{p,\rho};E_p).$$ 
The norm is given by the rescaled pull back to the disc. 
$$\|\beta\|_{k,\mu;p, \rho}:=\|\lambda_{p, \rho}^*\beta\|_{C^{k, \mu}(\Delta)},$$ where $\lambda_{p, \rho}:\Delta\rightarrow \Delta_{p, \rho}; \zeta\mapsto \rho \zeta e_p$, for any unit vector $e_p\in L_p$.
Thus $\mathcal C_{p, \rho}$ is an isomorphism from $\mathfrak B_{p,\rho}$ to $\mathfrak X_{p,\rho}$ and is the right inverse of
$$\bp: \mathfrak X_{p,\rho}\rightarrow \mathfrak B_{p,\rho}.$$

The map $\mathcal F_{p, \rho}$ can be viewed as a  $C^1$ map from a ball $B_\epsilon(0)\subset \mathfrak X$ into $\mathfrak B$. The linearization at $0$ is given by 
$$D\mathcal F_{p, \rho}|_0=\bp+K_{p, \rho},  $$
where the perturbation $K_{p, \rho}$ satisfies the operator norm bound
$$\|K_{p, \rho}\mathcal C_{p, \rho}\|\leq C\rho$$ 
Consequently,
\[
D\mathcal F_{p,\rho}(0)\mathcal C_{p,\rho}
=
I+K_{p,\rho}\mathcal C_{p,\rho}.
\]
For $\rho$ sufficiently small this operator is invertible, and
\[
D\mathcal F_{p,\rho}(0)^{-1}
=
\mathcal C_{p,\rho}
\left(I+K_{p,\rho}\mathcal C_{p,\rho}\right)^{-1}.
\]
It is straightforward to compute that 
$$\|\mathcal F_{p, \rho}(0)\|_{\mathfrak B}\leq C\rho^2,$$
$$\|\mathcal C_{p, \rho}(\mathcal F_{p, \rho}(0))\|_{\mathfrak B'}\leq C\rho^2$$
Applying the standard quantitative implicit function theorem, for $\rho>0$ sufficiently small we can find  $Y_{p, \rho}$ with $|Y_{p, \rho}|=O(|\zeta|^2)$ such that $\mathcal F_{p, \rho}(Y_{\rho,p})=0$ on $\Delta_{p, \rho}$ for all $p\in D$.  It is clear that the procedure can be made uniform in $p\in D$. 
\end{proof}

\section{The second diffeomorphism}
We now work under the identification given in Proposition \ref{p:vertical_disc}. 
We say a complex-valued differential form $\eta$ defined on an open subset of $L^\times$ is \emph{horizontal} if
$\iota_\xi\eta=\iota_{\bar\xi}\eta=0$.  
For a horizontal form $\eta$  we define
\[
 d_H\eta
 :=d\eta-\Theta\wedge\mathcal L_\xi\eta
          -\bar\Theta\wedge\mathcal L_{\bar\xi}\eta.
\]
Notice that for $m\geq 1$, we can lift elements in $\Omega^q(D, L^{-m}),\, q\geq 0$, as horizontal forms on $L^\times$ which are homogeneous of degree $m$ in the fiber direction.
\begin{lemma}
\label{lem:horizontal-differential-connection}
Let $\widehat\sigma$ be the fiberwise-homogeneous horizontal
lift of $\sigma\in\Omega^q(D,L^{-m})$. Then
\begin{equation}
        d_H\widehat\sigma
        =\widehat{d_{\nabla_{L^{-m}}}\sigma}.
 \label{eq:dH-degree-m-connection}
\end{equation}
\end{lemma}

\begin{proof}
Choose a local parallel unitary frame
$e$ of $L$.  Writing
$\sigma=\sigma_0\otimes e^{-m}$, its homogeneous lift is
\begin{equation}\label{eq-horizontal lift}
\widehat\sigma(p,w)=w^m\sigma_0(p).
\end{equation}
Since $\Theta=dw/w$, one has
$
 d(w^m)=m w^m\Theta.
$
Then by \eqref{eq-horizontal lift}, we obtain
\[
 \mathcal L_\xi\widehat\sigma=m\widehat\sigma,
 \qquad
 \mathcal L_{\bar\xi}\widehat\sigma=0.
\]
The induced flat connection on $L^{-m}$ is trivial in the frame
$e^{-m}$; hence the ordinary derivative of $\sigma_0$ is the coefficient
of $d_{\nabla_{L^{-m}}}\sigma$.  Therefore we have
\[
 d\widehat\sigma
 =m\Theta\wedge\widehat\sigma
  +\widehat{d_{\nabla_{L^{-m}}}\sigma}.
\]
The definition of $d_H$ therefore subtracts the first term, while the
$\bar\Theta$-term vanishes, and gives
\eqref{eq:dH-degree-m-connection}.
\end{proof}

Given Proposition \ref{p:vertical_disc}, the fact that each fiber is holomorphic in $J_M$ implies that 
$$\iota_{\bar\xi}\Omega_M=0. $$
Define
$$\alpha:=\iota_{\xi}\Omega_M, \qquad \beta:=\Omega_M-\Theta\wedge \alpha. $$
Then both $\alpha$ and $\beta$ are horizontal.  
Since $d\Omega_M=0$ and $d\Theta=0$, we have $\mathcal{L}_{\bar\xi}\Omega_M=0$ and  
\begin{equation}
\mathcal L_{\bar\xi}\alpha =\mathcal L_{\bar\xi}\beta=0,
 \qquad d_H\beta=0,
 \qquad \mathcal L_\xi\beta=d_H\alpha.
 \label{eq:closed-horizontal-identities}
\end{equation}
Notice that the function $z=-\log |\zeta|_h^2$ is $J_M$-pluriharmonic if $\Theta$ is of type $(1,0)$ with respect to $J_M$. The latter condition holds  precisely when $\Theta\wedge \Omega_M=0$, or equivalently $\beta=0$. 

We seek a fiber-preserving diffeomorphism that is holomorphic on
each fiber and eliminates the horizontal component $\beta$.
For $T$ sufficiently large, set
\[
    \mathcal U_T:=\{\zeta\in L^\times:z(\zeta)\geq T\}.
\]
For a complex-valued function $f$ on $\mathcal U_T$, sufficiently
small in $C^1$ with respect to the cylindrical metric $g$, define
\[
    \Phi_f:\mathcal U_T\longrightarrow\mathcal U_{T-1},
    \qquad
    \Phi_f(\zeta):=e^{f(\zeta)}\zeta.
\]
Indeed,
\[
    z\circ\Phi_f=z-2\operatorname{Re}f,
\]
so $\Phi_f(\mathcal U_T)\subset\mathcal U_{T-1}$ whenever
$\|\operatorname{Re}f\|_{L^\infty}<\tfrac12$.
Moreover we have
\[
        (\Phi_f)_*\bar\xi
        =(\mathcal L_{\bar\xi} f)\xi
         +(1+\mathcal L_{\bar\xi}\bar f)\bar\xi.
\]
Therefore if $f$ is holomorphic on each fiber, then $\Phi_f$ will preserve the holomorphicity of each fiber with respect to $\Omega_M$.

\subsection{The key Lemma}
 Let $\bL$ be a general holomorphic line bundle on $D$ endowed with a flat Hermitian metric. The Chern connection on $\bL$ can be decomposed as 
$$\nabla_{\bL}=\p_{\bL}+\bp_{\bL}. $$
Denote the Dolbeault Laplacian by
$$\Delta_{\bL}: =\bp_{\bL}\bp_{\bL}^*+\bp_{\bL}^*\bp_{\bL}.$$
We have natural Sobolev norms defined on $\Omega^{p, q}(D; \mathbb L)$:
$$\|s\|_{H^k}^2=\int_{D}\sum_{j=0}^k |\nabla_{\mathbb L}^js|^2\omega_D$$
For $q=0,1$, let
$$\mathcal H^{0,q}_{\mathbb L}:=\ker\bigl(\Delta_{\mathbb L}|_{\Omega^{0,q}(D,\mathbb L)}\bigr),$$
and let $\Pi^{0,q}_{\mathbb L}$ be the $L^2$-orthogonal projection onto $\mathcal H^{0,q}_{\mathbb L}$.
\begin{lemma}\label{l:flat_isometry}
For every $u\in \Omega^{0, 1}(D,\bL)$, there exists a unique $f\in \Omega^{0}(D,\bL)$ satisfying 
\[
    \bar\partial_{\bL} f = \iota_v\partial_{\bL}u, \quad f\perp_{L^2}\mathcal H^{0,0}_{\bL}.
\]
Moreover, for every $k\geq 0$,
\[
    \|f\|_{H^k} =\|(I-\Pi^{0,1}_{\bL})u\|_{H^k} \leq\|u\|_{H^k}.
\]
\end{lemma}

\begin{proof}
Denote $$P=\iota_v\p_{\bL}: \Omega^{0, q}(D, \bL)\rightarrow \Omega^{0, q}(D, \bL). $$
Since $v$ is parallel with unit norm and $\mathbb L$ is flat, it is easy to see that $P=\nabla_v, P^*=-\nabla_{\overline{v}},$ and $$\Delta_{\bL}=P^*P=PP^*. $$

Let $G_q$ be the inverse of $\Delta_{\mathbb L}$ on
$(\mathcal H_{\mathbb L}^{0,q})^\perp$, extended by zero on
$\mathcal H_{\mathbb L}^{0,q}$. Thus
\[
    \Delta_{\mathbb L}G_q =G_q\Delta_{\mathbb L} =\I-\Pi_{\mathbb L}^{0,q}.
\]
Since $P$ and $P^*$ commute with $\Delta_{\bL}$, they also commute with $G_q$. 
Given $u\in \Omega^{0, 1}(D, \bL)$, we define
$$f:=\bp_{\bL}^*G_1 Pu\in \mathrm{Im}(\dbar_{\mathbb L}^*)\subset (\mathcal H_{\bL}^{0,0})^\perp.$$
Since $Pu \in (\mathcal H_{\bL}^{0,1})^\perp$, we have
$$\bp_{\bL} f=\bp_{\bL} \bp_{\bL}^* G_1 Pu=(\I-\Pi_{\bL}^{0,1})Pu = Pu. $$
The uniqueness follows from the normalization that $f\in (\mathcal H_{\bL}^{0,0})^\perp$.

It is straightforward to compute that 
$$\|f\|_{L^2}^2=\langle\bp_{\bL}\bp^*_{\bL}G_1Pu, G_1Pu\rangle_{L^2}=\langle Pu, G_1Pu\rangle_{L^2}=\langle u,(\I-\Pi_{\bL}^{0,1})u\rangle_{L^2}=\|(\I-\Pi_{\bL}^{0,1})u\|_{L^2}^2. $$
Finally, since both the base metric and the connection on $\mathbb L$ are flat, covariant differentiation along a global parallel orthonormal frame commutes with the operators $P$, $\bp_{\bL}^*$, and $G_1$. From this one obtains the isometry for higher $k$.  
\end{proof}

\subsection{The linearized problem}
For $T\gg1$, we define $\mathcal U_T$ to be the subset in $L^\times$ given by $z\geq T$. Recall the exponential Sobolev weight, normalized at the left endpoint of the tail; compare \cite[\S1]{LockhartMcOwen1985}:
\begin{definition}
    For $ 0<\delta<\frac{1}{2}$, an integer $\ell\geq0$, and a tensor $\sigma$ on $\mathcal U_T$, define
\[
 \|\sigma\|_{\ell,\delta,T}:=\left(\sum_{j=0}^{\ell}\int_{\mathcal U_T}
   e^{2\delta(z-T)}|\nabla^j\sigma|_{g}^2\,d\mu_{g}\right)^{1/2}.
\]
We denote by $H^\ell_\delta(\mathcal U_T)$ the sections in $H^\ell_{\rm loc}$ for which this norm is finite. 
\end{definition}
\noindent Then for $\ell \geq 1$, we let
\[
 \mathcal H_{\delta,T}^\ell
 :=\left\{f\in H^\ell_\delta(\mathcal U_T;\C)\mid\mathcal L_{\bar\xi}f=0\right\}
\]
be the closed subspace of fiberwise-holomorphic functions and let
\[
 \mathcal D_{\delta,T}^\ell :=\{\gamma\in H^\ell_\delta (\mathcal U_T;\Lambda^2T^*\mathcal U_T\otimes\C)\mid d\gamma=0,\ \iota_{\bar\xi}\gamma=0\}.
\]

\begin{proposition}\label{p:weighted-right-inverse}
	There are $T_0>0$ and a bounded linear operator for $T\geq T_0, \ell\geq 1, \delta\in (0, \frac{1}{2})$, 
	$$\mathcal Q_T: \mathcal D_{\delta,T}^\ell \rightarrow \mathcal H_{\delta,T}^{\ell+1}$$
	such that $$d_H(\mathcal Q_T(\gamma)\vartheta)=\gamma-\Theta\wedge \iota_\xi\gamma, $$and  
	$$\|\mathcal Q_T(\gamma)\|_{\ell+1, \delta, T}\leq C(\ell, \delta)\|\gamma\|_{\ell, \delta, T}.$$ Moreover,  for every $f\in \mathcal H_{\delta,T}^{\ell+1}$, 
    $$\mathcal Q_T(df\wedge\vartheta)=f.$$
\end{proposition}

Let us explain this more clearly using the following commutative diagram:
\[
\begin{tikzcd}[column sep=5em, row sep=3.5em, labels={font=\small}]
\mathcal H^{\ell+1}_{\delta,T}
  \arrow[rr, "d(\,\cdot\,\wedge\vartheta)"]
  \arrow[dr, "d_H(\,\cdot\,\wedge\vartheta)"']
&& \mathcal D^{\ell}_{\delta,T}
  \arrow[ll, bend left=18, red, "\mathcal Q_T" red]
  \arrow[dl, "\text{take horizontal part}"]
\\
& H^{\ell}_{\delta}\bigl(\mathcal U_T;\Lambda^2 T^*\mathcal U_T\otimes\C\bigr) &
\end{tikzcd}
\]
The proposition asserts that the horizontal arrow admits a bounded left inverse
$\mathcal Q_T$ that is compatible with the other two maps, in the sense that
the resulting diagram remains commutative. Note that the horizontal map is nothing but $f \longmapsto \Phi_f^*\Omega_0-\Omega_0$ which is a natural
 linear approximation of the nonlinear operator
\[
f \longmapsto \Phi_f^*\Omega_M-\Omega_0.
\]
We may therefore apply the implicit function theorem to the functional
\[
f \longmapsto \mathcal Q_T\bigl(\Phi_f^*\Omega_M-\Omega_0\bigr),
\]
to get a zero of this functional. The commutativity property then implies that the
horizontal part of
\[
\Phi_f^*\Omega_M-\Omega_0
\]
vanishes, which implies the horizontal part of $\Phi_f^*\Omega_M$ vanishes.
\begin{proof}
 Write $\gamma=\Theta\wedge\alpha+\beta$ with $\alpha := \iota_\xi\gamma$. Since $\iota_{\bar \xi}\gamma=0$ and $d\gamma=0$, it follows that 
$\mathcal L_{\bar\xi}\gamma=0$. In particular, $\mathcal L_{\bar\xi}\alpha=\mathcal L_{\bar \xi}\beta=0$. 
We first carry out the construction for finite sums of Fourier modes.
 From Fourier expansion, 
$$\alpha=\sum_{m\geq 1} \alpha_m; \ \ \ \beta=\sum_{m\geq 1}\beta_m$$
where $\alpha_m, \beta_m$ are homogeneous of degree $m$ in the fiber direction. For different $m$, the homogeneous pieces are orthogonal in the weighted Sobolev
norms. We only have $m\geq 1$ since $\gamma$ has finite Sobolev norm. 
The property $\mathcal L_\xi\beta=d_H\alpha$ implies that 
\begin{equation}\label{eq--beta and alpha}
    m\beta_m=d_H\alpha_m. 
\end{equation}
Decompose the horizontal one-form $\alpha_m$ as $$\alpha_m=-u_m\vartheta+v_m ,$$ where $u_m\in \Omega^0(D, L^{-m})$, $v_m\in \Omega^{0,1}(D, L^{-m})$. Here we have used the identification between an element in $\Omega^p(D,L^{-m})$ with its fiberwise-homogeneous horizontal lift to $L^\times$.
 Then under this identification, using Lemma \ref{lem:horizontal-differential-connection}, \eqref{eq--beta and alpha} is equivalent to
\[
 m\beta_m
 =d_{\nabla_{L^{-m}}}\alpha_m
 =\vartheta\wedge\bar\partial_{L^{-m}}u_m
  +\partial_{L^{-m}}v_m.
\]
Then applying the Lemma~\ref{l:flat_isometry} to $L^{-m}$ gives a unique
$w_m\perp\mathcal H_{L^{-m}}^{0,0}$ with $$\bp_{L^{-m}} w_m=\iota_v\partial_{L^{-m}} v_m$$ and $\|w_m\|_{H^\ell}=\|(I-\Pi_{L^{-m}}^{0,1})v_m\|_{H^\ell}$. We define 
$$f_m:=-m^{-1}(u_m+w_m)$$
and 
$$f:=\sum_{m\geq 1}f_m.$$
Since
$\partial_{L^{-m}}v_m =\vartheta\wedge\iota_v\partial_{L^{-m}}v_m$ on $D$, we have
\begin{align*}
 d_H(f_m\vartheta)
 &=-\vartheta\wedge\bar\partial_{L^{-m}}f_m \\
 &=\frac1m\vartheta\wedge
   \bigl(\bar\partial_{L^{-m}}u_m+
         \bar\partial_{L^{-m}}w_m\bigr)
 =\beta_m.
\end{align*}
It follows that $d_H(f\vartheta)=\beta$. Then we define
$\mathcal Q_T(\gamma):=f$ and have 
\begin{equation}
    d_H(\mathcal Q_T(\gamma) \vartheta)=\beta=\gamma-\Theta\wedge \iota_{\xi}\gamma.
\end{equation}

To obtain the estimates on $f$ we first see that 
\begin{equation}\label{eq-vertical}
    \|\mathcal L_\xi f_m\|_{H^\ell}=\|mf_m\|_{H^\ell}\leq C\|\alpha_m\|_{H^\ell}, \quad \mathcal L_{\bar \xi}f_m=0.
\end{equation}
We also have $\|\bp_{L^{-m}} f_m\|_{H^\ell}=\|\iota_v\beta_m\|_{H^\ell}\leq \|\beta_m\|_{H^\ell}$, and 
since $L^{-m}$ is flat, by K\"ahler identity, we obtain 
$$\|\p_{L^{-m}} f_m\|_{H^\ell}=\| \bp_{L^{-m}} f_m\|_{H^\ell}\leq \|\beta_m\|_{H^\ell}.$$
Therefore we have 
\begin{equation}\label{eq:pure-horizontal-Q-estimate}
 \|\nabla_{L^{-m}}^{\ell+1}f_m\|_{L^2}
 \leq C_\ell\|\nabla_{L^{-m}}^\ell\beta_m\|_{L^2}.
\end{equation}
This controls the pure horizontal derivative of $f$ of order $\ell+1$.
	
Finally, we estimate the derivatives of \(f_m\) involving vertical directions. For a homogeneous function of degree $m$, each vertical derivative contributes a factor $m$ and each horizontal derivative becomes a
Chern covariant derivative on $L^{-m}$.  Thus
\begin{equation}\label{eq:vertical derivative}
\sum_{q=0}^j|\nabla^q f_m|_{g}^2 \leq C_j e^{-mz}\sum_{r+s\leq j} m^{2r}|\nabla_{L^{-m}}^sf_m|^2.
\end{equation}
The comparison constants depend only on $j$ and the fixed model
geometry, and are independent of $m$ and $T$.
Combining \eqref{eq-vertical}, \eqref{eq:pure-horizontal-Q-estimate} and \eqref{eq:vertical derivative}, multiplying by $e^{2\delta(z-T)}$, integrating over $\mathcal U_T$,
and summing over $m$ using the orthogonality, gives
\[
 \|f\|_{\ell+1,\delta,T} \leq C_{\ell,\delta}\|\gamma\|_{\ell,\delta,T}.
\]

For a general $\gamma\in\mathcal D_{\delta,T}^{\ell}$, the
fiberwise Fourier decomposition holds distributionally. Since
fiber rotations preserve the cylindrical metric, its connection,
and the weight, the Fourier partial sums converge in
$H^\ell_\delta(\mathcal U_T)$ and preserve the defining equations
of $\mathcal D_{\delta,T}^{\ell}$.

The coefficients of each homogeneous mode are Sobolev sections
on $D$. They can be regularized by the heat operators of the flat
Hodge Laplacians, which commute with
$d_{\nabla_{L^{-m}}}$ and therefore preserve
$m\beta_m=d_H\alpha_m$. Thus the preceding modewise construction
and estimates apply to these Sobolev coefficients by passage to
the limit.

Writing
\[
    \gamma^{(N)}:=\sum_{m=1}^N\gamma_m,
    \qquad
    f^{(N)}:=\sum_{m=1}^N f_m,
\]
the estimate and orthogonality give, for $N'>N$,
\[
    \|f^{(N')}-f^{(N)}\|_{\ell+1,\delta,T}^2
    \leq C_{\ell,\delta}^2
       \sum_{m=N+1}^{N'}\|\gamma_m\|_{\ell,\delta,T}^2.
\]
Hence $f^{(N)}$ converges in
$\mathcal H_{\delta,T}^{\ell+1}$ to a limit $f$.
The defining identity passes to the limit distributionally,
and the same estimate proves that
$\mathcal Q_T(\gamma):=f$ is a bounded linear operator.

Finally, let
$h=\sum_{m\geq1}h_m\in\mathcal H_{\delta,T}^{\ell+1}$
and take $\gamma=dh\wedge\vartheta$. Since
$\iota_\xi\vartheta=0$, its horizontal one-form component is
\[
    \alpha=\iota_\xi(dh\wedge\vartheta)
          =(\mathcal L_\xi h)\vartheta,
\]
so $\alpha_m=mh_m\vartheta$. In the preceding decomposition,
therefore, $u_m=-mh_m$ and $v_m=0$. The normalization in
Lemma~\ref{l:flat_isometry} gives $w_m=0$, and the construction
returns $f_m=h_m$. Consequently
\[
    \mathcal Q_T(dh\wedge\vartheta)=h
\]
and this completes the proof.
\end{proof}

\subsection{Applying the implicit function theorem}

Now we fix $k\geq 4$. Let
\[
\eta:=\Omega_M-\Omega_0.
\]
Then by Proposition \ref{p:vertical_disc}, we have $d\eta=0$, $\iota_{\bar\xi}\eta=0$, and
$|\nabla^j\eta|_g=O(e^{-z/2})$ for every $j\geq0$.
For $T\geq T_0+1$, define
\begin{equation}
\varepsilon_T
:=\|\eta\|_{k+2,\delta,T-1}
+\sum_{j=0}^{k+2}\sup_{\mathcal U_{T-1}}|\nabla^j\eta|_g.
\end{equation}
The asymptotic estimate \eqref{e:volume form close} and
$\delta<\frac12$ imply that
$\varepsilon_T=O(e^{-(1/2-\delta)T})$ and in particular
$\varepsilon_T\to0$ as $T\to\infty$.
\begin{proposition}\label{prop-nonlinear-solution}
There exists $\epsilon>0$ such that the map 
	$$\Psi_T: B_\epsilon(0)\subset \mathcal H_{\delta,T}^{k+1} \rightarrow \mathcal H_{\delta,T}^{k+1}; \quad f\mapsto \mathcal Q_T(\Phi_f^*\Omega_M-\Omega_0)$$
is a well-defined $C^1$ map of real Banach spaces. Furthermore, for any $T$ large enough, there is a unique
$f_T$ near the origin such that
$\Psi_T(f_T)=0$, and
\[
\|f_T\|_{k+1,\delta,T}\leq2C\varepsilon_T.
\]
\end{proposition}

\begin{proof}
Since $k\geq4$, weighted Sobolev embedding implies
that a sufficiently small ball in
$\mathcal H_{\delta,T}^{k+1}$ has uniformly small $C^1$ norm.  Hence
$\Phi_f$ is a $C^2$ diffeomorphism onto its image and maps $\mathcal U_T$
into $\mathcal U_{T-1}$.  Since $\mathcal L_{\bar\xi}f=0$, its
restriction to each fiber is holomorphic.  It follows that
$\Phi_f^*\Omega_M-\Omega_0$ belongs to
$\mathcal D^k_{\delta,T}$.

Since $\Phi_f^*\Omega_0-\Omega_0=df\wedge\vartheta$, by Proposition \ref{p:weighted-right-inverse}, we have
\[
    \mathcal Q_T(\Phi_f^*\Omega_0-\Omega_0) = \mathcal Q_T(df\wedge\vartheta) = f,
\]
therefore
\begin{equation*}
\Psi_T(f)=f+\mathcal E_T(f),
\qquad
\mathcal E_T(f):=\mathcal Q_T(\Phi_f^*\eta).
\end{equation*}
Fix $r_0>0$ sufficiently small, independently of $T$.
For
$f,g\in B_{r_0}(0)$, along the path $f_s=g+s(f-g)$ the derivative of the
pullback is a pullback of a Lie derivative of $\eta$ by a vertical
vector field whose coefficients are linear in $f-g$ and its complex
conjugate.  The Lie derivative contains at most one derivative of
$f-g$.  The uniform weighted multiplication and composition estimates
on the flat cylinder, Sobolev embedding, and estimates of $\eta$ therefore give
\begin{align*}
\|\Phi_f^*\eta\|_{k,\delta,T}&\leq C\varepsilon_T,\\
\|\Phi_f^*\eta-\Phi_g^*\eta\|_{k,\delta,T}&\leq C\varepsilon_T\|f-g\|_{k+1,\delta,T},\\
\|D(\Phi_{\bullet}^*\eta)_f[h]\|_{k,\delta,T}&\leq C\varepsilon_T\|h\|_{k+1,\delta,T},
\end{align*}
with constants $C$ independent of $T$. Applying Proposition~\ref{p:weighted-right-inverse} and increasing
$C$ if necessary, we obtain
\[
 \|\mathcal E_T(f)\|_{k+1,\delta,T}\leq C\varepsilon_T,
 \qquad
 \|\mathcal E_T(f)-\mathcal E_T(g)\|_{k+1,\delta,T}
 \leq C\varepsilon_T\|f-g\|_{k+1,\delta,T}.
\]
Choose $T$ sufficiently large that $C\varepsilon_T\leq\tfrac12$
and $2C\varepsilon_T<r_0$. Then the map
$f\mapsto-\mathcal E_T(f)$ is a contraction of the closed ball
\[
 \overline B_{2C\varepsilon_T}(0)
 \subset\mathcal H_{\delta,T}^{k+1}
\]
into itself. Its unique fixed point $f_T$ satisfies
\[
 \Psi_T(f_T)=0,
 \qquad
 \|f_T\|_{k+1,\delta,T}\leq2C\varepsilon_T.
\]
\end{proof}
We now fix a $T$ sufficiently large so that Proposition \ref{prop-nonlinear-solution} holds. For this $f_T$, write
\[
    \Phi_{f_T}^*\Omega_M=\Theta\wedge\alpha_T+\beta_T,
\]
where $\beta_T$ is horizontal. 
Then we obtain the following theorem.

\begin{theorem}\label{thm--existence of plh function}
   We have $\beta_T=0$ and hence the function $z=-\log |\zeta|^2_h$ is $\Phi_{f_T}^*(J_M)$-pluriharmonic.
\end{theorem}
\begin{proof}
    Proposition~\ref{p:weighted-right-inverse} and Proposition~\ref{prop-nonlinear-solution} give
\[
0=d_H\bigl(\mathcal Q_T(\Phi_{f_T}^*\Omega_M-\Omega_0)\vartheta\bigr)
=\Phi_{f_T}^*\Omega_M-\Omega_0
-\Theta\wedge\iota_\xi(\Phi_{f_T}^*\Omega_M-\Omega_0).
\]
Since $\Omega_0$ has no horizontal
component, we obtain
\[
\beta_T=\Phi_{f_T}^*\Omega_M-\Theta\wedge\iota_\xi\Phi_{f_T}^*\Omega_M=0.
\]
Since $\Phi_{f_T}^*\Omega_M$ is a nowhere vanishing $(2,0)$-form
for $\Phi_{f_T}^*J_M$, the identity
$\Theta\wedge\Phi_{f_T}^*\Omega_M=0$ implies that $\Theta$ is
of type $(1,0)$ for this complex structure. Together with
\[
    dz=-(\Theta+\bar\Theta),\qquad d\Theta=0,
\]
this proves the asserted pluriharmonicity.
\end{proof}

\section{Completion of the proof}\label{sek}
\subsection{Proof of Theorem \ref{thm:dps-semipositivity}}
Theorem~\ref{thm--existence of plh function} gives us that the function $z=-\log|\zeta|^2_h$ is $\Phi_{f_T}^*(J_M)$-pluriharmonic. Returning to the original surface, let $\Upsilon$ denote the
identification constructed in Proposition~\ref{p:vertical_disc},
and set
\[
    A:=\Upsilon\circ\Phi_{f_T}.
\]
Since $f_T$ is uniformly bounded, the image of $\Phi_{f_T}$
contains a smaller punctured disk bundle. Thus, after shrinking
a neighborhood $U$ of $D$, the function
\[
    r:=z\circ A^{-1}
\]
is defined on $U\setminus D$.

Since $k\geq4$, Sobolev embedding gives $A\in C^2$ on the
interior of the tail. Hence $r\in C^2(U\setminus D)$, and Theorem \ref{thm--existence of plh function} gives
\[
    \partial_{J_M}\bar\partial_{J_M}r=0.
\]
Taking the trace with respect to a smooth local Hermitian metric
and applying elliptic regularity shows that $r$ is smooth.

Finally, in a local parallel fiber coordinate $w$, the first
identification satisfies
\[
    w_M\circ\Upsilon(p,w)=c(p)w+O(|w|^2),
    \qquad c(p)\neq0,
\]
for every local holomorphic defining function $w_M$ of $D$.
Since $|e^{f_T}|$ is uniformly bounded above and below, this gives a constant $C>0$ satisfying
\[
   C^{-1} |w|\leq |w_M\circ A(p,w)|\leq C |w|.
\]
Consequently, the function $r$ is defined on $X=M\setminus D$ on a punctured neighborhood of $D$, is pluriharmonic with respect to $J_M$, and 
\[
    r+\log|w_M|^2
\]
is locally bounded near $D$. Hence $r+\log|w_M|^2$ extends to a pluriharmonic function across $D$, and by the Poincar\'e-Lelong formula, $F:=-r$ (extended by $-\infty$ on $D$) is a plurisubharmonic function on some neighborhood $U$ of $D$ in $M$, which is comparable to $\log|w_M|^2$ and satisfies $dd^c F=[D]$. The first conclusion of Theorem \ref{thm:dps-semipositivity} then follows from Proposition \ref{ued}, (2)$\Rightarrow$(3).

To deduce from this that $\mathcal{O}_M(D)$ is Hermitian semipositive, we recall the simple argument from \cite[p.442]{brunella}. Adding a constant to $F$, we may assume that $\{F=1\}\subset U$ is nonempty, and we can replace $U$ with $\{F<1\}$. We then pick a smooth concave function $\varphi:\mathbb{R}\to\mathbb{R}$ with $\varphi(t)=0$ for $t\geq 0$ and $\varphi(t)=t+c$ for $t\leq -1$, for some constant $c$, then the function $\tilde{F}:M\to\mathbb{R}\cup\{-\infty\}$ given by $\widetilde{F}=\varphi\circ F$ on $U$, and identically zero outside $U$, is smooth and plurisuperharmonic on $U\backslash D$, and still satisfies
\begin{equation}
dd^c \widetilde{F}=[D],
\end{equation}
in some neighborhood of $D$, still denoted by $U$. Pick a defining section $s_D$ of $D$, and then
 on $X$ we define
\[
        |s_D|_h^2=e^{\widetilde F}.
\]
From this, it is immediate to see that this defines a smooth Hermitian metric on the line bundle over $M$ with semipositive curvature. See also the proof of Proposition \ref{ued}.

\subsection{Proof of Theorem \ref{thm:surface-semipositivity-classification}}
As we will see later, the key point is to classify rational surfaces with semipositive anticanonical bundle. In particular, \(K^{-1}_X\) is nef. If \(K^{-1}_X\) is also big, then it is semiample by Kawamata's basepoint-free theorem. We will therefore focus on the remaining case in which $X$ is rational and $K^{-1}_X$ is nef with $K_X^2=0.$

Following Sakai \cite{Sakai2001}, we call a smooth projective rational surface $X$ a
\emph{generalized Halphen surface} if it admits an effective
anticanonical divisor
\[
       D=\sum_i m_iD_i\in|K^{-1}_X|,\qquad m_i>0,
\]
of canonical type, which means that $D\mathbin{\cdot}D_i=0$ for every
irreducible component $D_i$
\cite[\S4, Definitions~1--2,
p.~180]{Sakai2001}.  Whenever needed, we fix such a divisor $D$. We use
Sakai's elliptic, multiplicative, and additive terminology
\cite[\S4 and Table~2]{Sakai2001}.
A \emph{Halphen surface of index $m$} is a generalized Halphen surface for which $\dim |-mK_X|>0$, and $m$ is the least positive integer such that this happens.

The following proposition shows that the surfaces we are interested in satisfy the generalized Halphen condition.

\begin{proposition}
\label{prop:nef-anticanonical-halphen}
Let $X$ be a smooth projective rational surface.  The following are
equivalent:
\begin{enumerate}
\item[\textup{(i)}] $K^{-1}_X$ is nef and $K_X^2=0$;
\item[\textup{(ii)}] $X$ is a generalized Halphen surface.
\end{enumerate}
When these conditions hold, $\vert K^{-1}_X\vert\neq\varnothing$, and every effective
anticanonical divisor on $X$ is of canonical type.  
\end{proposition}

\begin{proof}
If $D=\sum_i m_iD_i$ is of canonical type, then
$D^2=\sum_i m_i(D\mathbin{\cdot}D_i)=0$, and $D$ is nef because it has
zero intersection with its components and nonnegative intersection with
every other irreducible curve.  Conversely, if $K^{-1}_X$ is nef and
$K_X^2=0$, Riemann--Roch and Serre duality on the rational surface $X$
give $h^0(X,K^{-1}_X)\geq1$.  For any
$D=\sum_i m_iD_i\in|K^{-1}_X|$, all $D\mathbin{\cdot}D_i$ are nonnegative
and their weighted sum is $D^2=0$; hence they all vanish.  
\end{proof}

Using Sakai's classification together with Koike's results \cite{Koike2023}, we classify all generalized Halphen surfaces whose anticanonical line bundles are semipositive.
\begin{proposition}
\label{prop:generalized-halphen-criterion}
Let $X$ be a generalized Halphen surface, fix an anticanonical divisor
$D$ of canonical type.  Then
$K_X^{-1}=\mathcal{O}_X(D)$ is nef of numerical dimension one, and the following
assertions hold.
\begin{enumerate}
\item[\textup{(A)}] In the case of elliptic type, the divisor $D$ is a
smooth elliptic curve and $K_X^{-1}$ is always Hermitian semipositive.
It is semiample exactly when $N_{D}$ is torsion.
\item[\textup{(B)}] In the case of multiplicative types,
the divisor $D$ is a reduced cycle of rational curves and
$\operatorname{Pic}^0(D)\simeq\C^*$.  With $q_D$ defined as in
Theorem~\ref{thm:surface-semipositivity-classification},
\[
 K_X^{-1}\text{ is Hermitian semipositive}
 \quad\Longleftrightarrow\quad |q_D|=1.
\]
It is semiample exactly when $q_D$ is a root of unity.
\item[\textup{(C)}] In the case of additive types,
the following are equivalent:
\begin{enumerate}
\item[\textup{(i)}] $K_X^{-1}$ is Hermitian semipositive;
\item[\textup{(ii)}] $K_X^{-1}$ is semiample;
\item[\textup{(iii)}] $X$ is a Halphen surface of index one, so
$|K^{-1}_X|$ is a basepoint-free genus-one pencil.
\end{enumerate}
\end{enumerate}
\end{proposition}

\begin{proof}[Proof of Proposition~\ref{prop:generalized-halphen-criterion}]
The list of elliptic, multiplicative, and additive types is Sakai's
classification \cite[\S4 and Table~2]{Sakai2001}.  Since $D$ is of
canonical type, it is nef, $D^2=0$, and $\operatorname{nd}(\mathcal{O}_X(D))=1$.  Hence
\cite[Theorem~1.1]{Koike2026} gives
\begin{equation}
 \label{eq:halphen-koike-criterion}
 \mathcal{O}_X(D)\text{ is Hermitian semipositive}
 \quad\Longleftrightarrow\quad
 \mathcal{O}_X(D)|_V\text{ is unitary flat for some neighborhood }V\supset D.
\end{equation}

We first use a short exact sequence argument to show that if  $D$ is
either a smooth elliptic curve or a reduced rational cycle, $\mathcal{O}_X(D)$ is semiample if and only if $N_D$ is torsion. Suppose
$N_D:=\calO_D(D)$ is torsion of order exactly $m$.  Then duality gives
$H^0(D,N_D^k)=H^1(D,N_D^k)=0$ for $1\leq k<m$.  The sequences
\begin{equation}\label{esatta}
 0\longrightarrow\calO_X((k-1)D)
 \longrightarrow\calO_X(kD)
 \longrightarrow N_D^k\longrightarrow0
\end{equation}
and $H^1(X,\calO_X)=0$ then give, by induction,
$h^0(X,kD)=1$ and $H^1(X,kD)=0$ for $0\leq k<m$.  For $k=m$, the
constant section of $N_D^m\simeq\calO_D$ lifts; together with $s_D^m$
it generates $\calO_X(mD)$.  Thus $|mD|$ is basepoint-free and no
smaller multiple moves.  Conversely, let $D$ be a connected divisor of canonical type,
possibly nonreduced. If $\mathcal{O}_X(D)$ is semiample, then for some $n>0$
the line bundle
\[
    N_D^n=\calO_D(nD)
\]
is globally generated. Let
$\varphi:D\to\mathbb P^r$ be the morphism defined by a set of
generating sections.
Since
\[
    \deg(N_D^n|_{D_i})=nD\mathbin{\cdot}D_i=0
\]
for every irreducible component $D_i$ of $D_{\mathrm{red}}$,
the restriction of $\varphi$ to each component is constant.
Connectedness therefore implies that
$\varphi(D_{\mathrm{red}})$ is a single point.
Choose a homogeneous linear form on $\mathbb P^r$ that does not
vanish at this point. Its pullback is a section of $N_D^n$
which is nowhere zero on $D_{\mathrm{red}}$.
By Nakayama's lemma, this section generates $N_D^n$ at every
point of $D$, including when $D$ is nonreduced.
Hence $N_D^n\simeq\calO_D$.

In elliptic type, $D$ is a smooth elliptic curve and
Theorem~\ref{thm:dps-semipositivity} gives semipositivity, while the
preceding argument gives semiampleness exactly when $N_D$ is torsion.

In multiplicative type, $D$ is a reduced cycle of rational curves and
$\operatorname{Pic}^0(D)\simeq\C^*$, so
\eqref{eq:halphen-koike-criterion} makes $|q_D|=1$ necessary.  For the
converse, we apply \cite[Theorem~3.1]{Koike2023} through the general
moving-blow-up configuration of \cite[\S4.2]{Koike2023}.

Since the surface $X$ is not minimal, we
choose a $(-1)$-curve $E$.  Then we have
\[
        D\mathbin{\cdot}E=-K_X\mathbin{\cdot}E=1.
\]
The curve $E$ cannot be a component of $D$ because of the canonical-type
condition.  Because $D$ is a
reduced nodal cycle, total intersection number one forces $E$ to meet
$D$ transversely at a single smooth point.
Let $\rho:X\to Z$ contract $E$, put $p:=\rho(E)$, and set
$D_Z:=\rho(D)$.  Then $D$ is the strict transform of $D_Z$, which is smooth at $p$.  Thus $\rho|_D:D\to D_Z$ is an isomorphism and
\[
 K_X=\rho^*K_Z+E,\qquad \rho^*D_Z=D+E,
\]
so $D_Z\in|K^{-1}_Z|$.  Write $D_1$ for the component meeting $E$ and
$C_i:=\rho(D_i)$.  Then $\rho^*C_1=D_1+E$, while
$\rho^*C_i=D_i$ for $i\ne1$.  Since
$D\mathbin{\cdot}D_i=0$ for every $i$, we obtain
\[
 D_Z\mathbin{\cdot}C_1
   =1,\quad
 D_Z\mathbin{\cdot}C_i
   =0\qquad(i\ne1).
\]
Thus $L:=N_{D_Z/Z}$ has multidegree $(1,0,\ldots,0)$.  Restricting
$\rho^*D_Z=D+E$ to $D$ also gives, under $D\simeq D_Z$,
\[
        N_D\simeq L\otimes\calO_{D_Z}(-p).
\]
Put $\Sigma:=C_1\setminus\operatorname{Sing}(D_Z)$.
Since $L$ has multidegree $(1,0,\ldots,0)$, one has
$h^0(D_Z,L)=1$, and its nonzero section has a unique simple
zero $a\in\Sigma$. Choose the coordinate $s$ on the normalization
of $C_1$ which sends the two attaching points to $0,\infty$
and $a$ to $1$, thereby identifying $\Sigma$ with $\C^*$.
Write $p(s)\in\Sigma$ for the point with coordinate $s$.
Koike shows that
\[
 \beta:\C^*\longrightarrow\operatorname{Pic}^0(D_Z),
 \qquad
 \beta(s)=L\otimes\calO_{D_Z}(-p(s)),
\]
is an algebraic isomorphism with $\beta(1)=\calO_{D_Z}$
\cite[\S4.2, pp.~259--260]{Koike2023}.

Using the identification $\Sigma\simeq\C^*$, let
\[
 \Gamma:=\{(p(s),s):s\in\Sigma\}\subset Z\times\Sigma,
 \qquad
 \mathcal X:=\operatorname{Bl}_{\Gamma}(Z\times\Sigma).
\]
Let $\mathcal Y$ be the strict transform of $D_Z\times\Sigma$.
Since $\Gamma$ is Cartier in $D_Z\times\Sigma$, the blow-up induces
an isomorphism $\mathcal Y\simeq D_Z\times\Sigma$.
Thus this is the configuration of \cite[\S4.2]{Koike2023}.
The fibers of $\mathcal X\to\Sigma$ satisfy
\[
        X_s=\operatorname{Bl}_{p(s)}Z,
        \qquad
        Y_s\in|K^{-1}_{X_s}|,
        \qquad
        N_{Y_s/X_s}\simeq\beta(s),
\]
where $Y_s$ is the strict transform of $D_Z$. We restrict this family to the open neighborhood
$S\subset\C^*$ of $U(1)$ used in \cite[\S4.2]{Koike2023}.
For the parameter $s_0$ with $p(s_0)=p$, one has
$(X,D)\simeq(X_{s_0},Y_{s_0})$ and $N_D\simeq\beta(s_0)$.
After transporting $\alpha_D$ through $D\simeq D_Z$, the automorphism
$\alpha_D\circ\beta$ of $\C^*$ fixes $1$, so it is $s\mapsto s$ or
$s\mapsto s^{-1}$.  Consequently
\[
        |q_D|=1\quad\Longleftrightarrow\quad s_0\in U(1),
\]
and the torsion order of $q_D$ equals that of $s_0$.

Assumption~5 requires $(Y_s,X_s)$ to be of infinite type for
every $s\in U(1)$. For non-torsion $s$, this follows from
\[
 H^1(Y_s,N_{Y_s/X_s}^{-n})=0
 \qquad(n\geq1),
\]
as explained in \cite[\S4.2.5]{Koike2023}.
For torsion parameters, we verify the sufficient condition
in Assumption~$5''$: a positive multiple of $Y_s$ is a fiber
of a proper holomorphic map on a neighborhood of $Y_s$.
Let $s$ have exact order $m$.
The surface $Z$ is rational, hence
every $X_s$ is rational; moreover, $Y_s$ is a reduced rational cycle
with $\omega_{Y_s}\simeq\calO_{Y_s}$.  The exact-sequence argument
above, applied to $(X_s,Y_s)$, gives
\[
 h^0(X_s,kY_s)=1\quad(1\leq k<m),
 \qquad
 h^0(X_s,mY_s)=2,
\]
and $|mY_s|$ is basepoint-free.  Let $t_s$ be the lift of
$1\in H^0(Y_s,\calO_{Y_s})$ furnished by the restriction sequence.
The two sections $s_{Y_s}^m,t_s$ form a basis and define
\[
        \phi_s=[s_{Y_s}^m:t_s]:X_s\longrightarrow\mathbb P^1.
\]
Consequently
\[
        \phi_s^*\{[0:1]\}=\operatorname{div}(s_{Y_s}^m)=mY_s
\]
scheme-theoretically. Restricting $\phi_s$ over a sufficiently small
disc $B_s$ around $[0:1]$, with that point identified with $0$, gives
a surjective proper holomorphic map
\[
        b_s:V_s:=\phi_s^{-1}(B_s)\longrightarrow B_s,
        \qquad b_s^*\{0\}=mY_s,
\]
where $V_s$ is a neighborhood of $Y_s$.  This is precisely
Assumption~$5''$ of \cite[\S4.2, p.~260]{Koike2023}.

Koike verifies Assumptions~1--4 for this moving blow-up family in
\cite[\S\S4.2.1--4.2.4, pp.~260--262]{Koike2023}.
The verification of Assumption~$5''$ at the torsion parameters above,
together with the non-torsion argument in
\cite[\S4.2.5, p.~263]{Koike2023}, establishes Assumption~5.
Thus \cite[Theorem~3.1, pp.~252--253]{Koike2023} gives, for every
$s\in U(1)$, local defining functions $\widehat w_{j,s}$ for $Y_s$
satisfying
\[
    \widehat w_{j,s}=t_{jk}(s)\widehat w_{k,s}.
\]
For this family, the conormal transition constants in
\cite[\S4.2]{Koike2023} satisfy
\[
    t_{jk}(s)\in\{1,s,s^{-1}\}.
\]
Consequently, when $s\in U(1)$, all these transition constants
have modulus one, so $[Y_s]$ is unitary flat on a neighborhood
of $Y_s$. Since $|q_D|=1$ implies $s_0\in U(1)$ and
$(X,D)\simeq(X_{s_0},Y_{s_0})$, it follows that $\mathcal{O}_X(D)$ is
unitary flat on a neighborhood of $D$. By
\eqref{eq:halphen-koike-criterion}, $\mathcal{O}_X(D)$ is therefore
Hermitian semipositive.

Finally, in additive type put $Y=D_{\mathrm{red}}$.  For every additive
Sakai type, the reduced support $Y$ has no topological cycle; equivalently,
$H_1(Y,\mathbb Z)=0$ \cite[\S4, p.~181]{Sakai2001}.  Thus every
unitary-flat line bundle on $Y$ has trivial monodromy and is holomorphically
trivial.
Consequently \cite[Theorems~1.1 and~1.2]{Koike2026} show that $\mathcal{O}_X(D)$ is
Hermitian semipositive exactly when it is semiample.  In this type the
additive period group is
$\operatorname{Pic}^0(D)\simeq(\C,+)$ \cite[\S4]{Sakai2001}.  If
$\mathcal{O}_X(D)$ is semiample, the preceding argument makes
$N_D:=\calO_D(D)$ torsion, hence trivial.  Since
$H^1(X,\calO_X)=0$, the sequence
\[
 0\longrightarrow\calO_X\longrightarrow\calO_X(D)
 \longrightarrow N_D\simeq\calO_D\longrightarrow0
\]
lifts $1\in H^0(D,\calO_D)$.  This lift and $s_D$ generate
$\calO_X(D)$.  Since a minimal rational surface has positive
$K^2$, the surface $X$ has a $(-1)$-curve $E$; adjunction gives
$D\mathbin{\cdot}E=1$.  Hence $|K^{-1}_X|=|D|$ is a basepoint-free
genus-one pencil and $X$ has index one.
\end{proof}

\begin{proof}[Proof of
Theorem~\ref{thm:surface-semipositivity-classification}]
First we suppose that we are in one of cases (A)--(D), and we show that $K_M^{-1}$ is Hermitian semipositive.

Case~\textup{(A)} follows by pulling back a Fubini--Study metric from a
globally generated power and taking its root.  Cases~\textup{(B)} and
\textup{(C)} follow from
Proposition~\ref{prop:generalized-halphen-criterion}.

For~\textup{(D)}, write $\pi:M\to B$ and $E:=\calO_B\oplus L$.  Since
$E$ is unitary flat, the relative Fubini--Study metric and
\[
 K_M^{-1}\simeq
 \calO_{\mathbb P_B(E)}(2)\otimes\pi^*(\det E)^{-1}
\]
give a semipositive metric.  Moreover,
\[
 h^0(M,K_M^{-m})
 =\sum_{j=0}^{2m}h^0(B,L^{j-m})=1
 \qquad(m\geq1),
\]
so $K_M^{-1}$ is not semiample.\\

Conversely, we suppose that $K_M^{-1}$ is Hermitian semipositive but not
semiample, and we wish to show that we must be in cases (B)--(D).  

Clearly, $K^{-1}_M$ is nef.  If $\kappa(M)\geq0$, then 
$mK_M\simeq\calO_M$, a contradiction.  Thus $\kappa(M)=-\infty$.
The Enriques--Kodaira classification
\cite[Chapter~VI, Theorem~1.1]{BHPV2004} now makes $M$ rational or a
blow-up of a ruled surface, and in particular projective.

In either case $K^{-1}_M\not\equiv0$.  Hence our nef,
non-semiample $K^{-1}_M$ is bad in the sense of
\cite[Definition--Proposition~1.4]{BauerPeternell2004}, and the
classification \cite[Proposition~1.6]{BauerPeternell2004} leaves two
possibilities.  In the rational case, Propositions
\ref{prop:nef-anticanonical-halphen} and
\ref{prop:generalized-halphen-criterion} give precisely
\textup{(B)} and~\textup{(C)}.  Otherwise
$M\simeq\mathbb P_B(E)$ over an elliptic curve and, after twisting $E$,
\[
       0\longrightarrow\calO_B\longrightarrow E
       \longrightarrow L\longrightarrow0,
       \qquad L\in\operatorname{Pic}^0(B),
\]
where either $L\simeq\calO_B$ and the extension is non-split, or $L$ is
non-torsion.  The non-split extension with $L\simeq\calO_B$ has nef but
non-semipositive anticanonical bundle
\cite[Corollary~2.5.2]{DPS2001}, so it cannot occur here.  If $L$ is
non-torsion, then
\[
 \operatorname{Ext}^1(L,\calO_B)\simeq H^1(B,L^{-1})=0,
\]
and the extension splits.  We obtain precisely case~\textup{(D)}.

Finally, we show that cases (A)--(D) are all distinct. In cases \textup{(B)}--\textup{(D)} we have that $K_M^{-1}$ is not semiample, and in case
\textup{(D)} we have that $M$ is not rational.  In \textup{(B)} and~\textup{(C)}, the
restriction sequence and $H^0(D,N_D)=0$ show that $D$ is the unique
anticanonical divisor; since it is smooth elliptic in case (B) and a cycle in case (C), the four alternatives are indeed pairwise disjoint.
\end{proof}

\subsection{Proof of Remark \ref{rem:non-kahler-surfaces}}\label{sec:proof of remark}

Indeed, suppose that $K_M^{-1}$ admits a smooth Hermitian metric $h$ with nonnegative curvature $\alpha:=R_h\geq 0.$
Choose a Gauduchon metric with fundamental form $\omega$ \cite{Gauduchon1977}.  If
$\kappa(M)\geq0$, take $m>0$ and a nonzero section of $mK_M$, with zero
divisor $C$.  Poincar\'e--Lelong and
$\partial\bar\partial\omega=0$ give
\[
 0\leq\int_C\omega
   =\deg_\omega(mK_M)
   =-m\int_M\alpha\wedge\omega\leq0.
\]
Thus $C=0$, $mK_M\simeq\calO_M$, and $\alpha=0$.  The non-K\"ahler
surfaces in this case are precisely the primary and secondary Kodaira
surfaces; see \cite[Chapter~V, \S5 and Chapter~VI,
Theorem~1.1]{BHPV2004}.  Conversely, their canonical bundles are torsion
and hence Hermitian flat.

If $\kappa(M)=-\infty$, then $M$ is of class~VII.  We then have $c_1(M)^2=-b_2(M)$ and hence $b_2(M)=0$ and $M$ is minimal.  Bogomolov's classification now makes
$M$ a Hopf or an Inoue surface \cite{Bogomolov76,LYZ90, LYZ94,Teleman94}.  The Inoue case is
impossible by \cite[Remark 4.2]{teleman}.  

The Hopf case can be extracted from \cite[Proof of Proposition 6.4]{DPS1994} and \cite[Proposition~3.2]{ChioseToma2023}, for convenience we give the short argument.
Every Hopf surface admits a finite unramified primary cover.
Hermitian semipositivity pulls back and descends under finite
unramified covers; for descent, pass to a finite unramified Galois
cover that factors through the given cover and take the geometric
mean of the deck-translated metrics.

For a diagonal primary Hopf surface
\[
 M=(\C^2\setminus\{0\})/\langle g\rangle,\qquad
 g(z_1,z_2)=(\alpha_1z_1,\alpha_2z_2),\qquad
 a_i=-\log|\alpha_i|>0.
\]
Choose a smooth convex regularized maximum $\widetilde{\max}(s,t)$,
equal to $\max(s,t)$ for $|s-t|\geq1$, satisfying
$\widetilde{\max}(s+c,t+c)=\widetilde{\max}(s,t)+c$. Set
\[
 \varphi=(a_1+a_2)\widetilde{\max}\left(
 a_1^{-1}\log|z_1|^2,\,
 a_2^{-1}\log|z_2|^2\right).
\]
This function is plurisubharmonic and extends smoothly across
each coordinate axis away from the origin, since there it agrees
with the branch involving the nonvanishing coordinate. Moreover,
\[
 \varphi\circ g=\varphi-2(a_1+a_2)
              =\varphi+\log|\det Dg|^2.
\]
Thus the metric assigning squared norm $e^{-\varphi}$ to
$\partial_{z_1}\wedge\partial_{z_2}$ descends to a smooth
semipositive metric on $K_M^{-1}$. This includes all elliptic primary Hopf
surfaces: they are precisely the diagonal ones with
$\alpha_1^p=\alpha_2^q$ for some positive integers $p,q$
\cite[\S3.1.2]{ChioseToma2023}.

A non-diagonal primary Hopf surface is thus non-elliptic and hence of class $0$, and so the obstruction proved in \cite[Proposition~3.2]{ChioseToma2023} shows that the only closed semipositive $(1,1)$-form is the zero form. Since $K_M^{-1}$ is not Hermitian flat (see e.g. \cite[p.149]{TW}), this shows that these surfaces cannot have $K_M^{-1}$ Hermitian semipositive. See also \cite{WuHopf2026} for a detailed discussion for semipositive line bundles on Hopf manifolds.

\subsection{Proof of  Theorem \ref{ueda}}
By Proposition \ref{ued}, the condition that the embedding $D\subset M$ is vertically linearizable is equivalent to $\mathcal{O}_M(D)$ being Hermitian semipositive. Also, it is clearly implied by the stronger condition that the embedding be linearizable, i.e. that $D\subset M$ admits a holomorphic tubular neighborhood.

Our assumption is that the elliptic curve $D$ is embedded in a compact complex surface $M$, with $\deg N_{D/M}=0$, and with the embedding formally vertically linearizable (equivalently, all Ueda classes $u_n(D,M)\in H^1(D,N_{D/M}^{-n})$ of this embedding vanish).

Our goal is to show that the embedding $D\subset M$ is vertically linearizable. If $N_{D/M}$ is torsion then this follows from \cite[Theorem 3]{Ueda}, so we assume also that $N_{D/M}$ is nontorsion. Then by \cite[Prop.3.1.2]{AD} we have that $\kappa(M)=-\infty,$ so by the Kodaira-Enriques classification $M$ must be rational, birationally ruled and non-rational, or of class VII.

Observe that the same argument as in \eqref{esatta} gives that $h^0(M,\mathcal{O}_M(mD))=1$ for every $m\geq 0$, so in particular the divisor $D$ in $M$ has Iitaka dimension zero. 

{\bf Case 1. }Assume first that $M$ is rational. Then \cite[Proposition 5]{Sakai} gives that either $D\sim K^{-1}_M$, and so we can apply Theorem \ref{thm:dps-semipositivity}, or else there is a $(-1)$-curve in $M$ disjoint from $D$. Contracting it, we can repeat this argument, which has to stop after finitely many such contractions.

{\bf Case 2. }Now assume that $M$ is non-rational and is birational to a ruled surface, so there is a sequence of point blowups $M\to M'$ and $M'=\mathbb{P}_B(V)\to B$ is a ruled surface. As noted in \cite{AD}, the base $B$ must be elliptic. Repeatedly contracting any $(-1)$-curve which is disjoint from $D$, and replacing $D$ with its image, we obtain another sequence of point blowups $M\to M''$. The morphism $f:M\to B$ restricted to the exceptional divisor of the last blowup of $M\to M''$ must be constant (it's a holomorphic map from $\mathbb{P}^1$ to an elliptic curve), so $f$ factors through this blowup. Repeating this, we see that the morphism $M\to M''$ is over $B$, and the resulting morphism $f'':M''\to B$ has generic fiber $\mathbb{P}^1$. Thus \cite[Proposition 6]{Sakai} applies to it, and shows that the image $D''$ of $D$ in $M''$ is a section of $f''$. We claim that $M''$ is minimal, hence of the form $M''=\mathbb{P}_B(W)\to B.$ Indeed, let $\nu:M''\to M_{\rm min}$  contract all vertical $(-1)$-curves. If $\nu$ was not the identity, then we can easily see that $M''$ would contain a $(-1)$-curve disjoint from $D''$, which is absurd. Indeed, consider the first blowup in $\nu$. If the center of the blowup does not lie on $\nu(D'')$ then we are done, and if it does then the strict transform of the fiber of $M_{\rm min}$ through that point is a $(-1)$-curve disjoint from the strict transform of $\nu(D'')$, and similarly for the next blowups.

So far we have shown that $M''=\mathbb{P}_B(W)\to B$. The section $D''$ corresponds to a quotient sequence
$$0\to Z\to W\to Y\to 0,$$
where $Z,Y$ are line bundles on $B$. The normal bundle of $D''$ is then isomorphic to $Z^{-1}\otimes Y$. The extension class of this sequence lies in $H^1(B,Z\otimes Y^{-1})=H^1(D'',N_{D''/M''}^{-1})=0,$ so the extension splits and $W=Z\oplus Y$ and $M''=\mathbb{P}_B(\mathcal{O}_B\oplus N_{D''/M''}).$ In particular, $D''\subset M''$ is linearizable, hence so is $D\subset M$.

{\bf Case 3. }Lastly, assume $M$ is class VII. In this case, \cite[Proposition 3.3.1]{AD} shows that $M$ is a diagonal nonresonant primary Hopf surface, a secondary Hopf surface which is a finite cyclic quotient of the previous one, or a blowup of these at points away from $D$. These blowups don't affect vertical linearizability, so it suffices to consider the diagonal nonresonant primary and secondary Hopf cases. In the first case, $M$ contains precisely $2$ elliptic curves, and by direct inspection their embeddings are linearizable, see \cite[\S2.2.6]{AD}, and the second case can be reduced to the first via the cyclic quotient \cite[\S2.3.1]{AD}:
on the finite covering Hopf surface, a neighborhood of each lifted
elliptic curve is a neighborhood of the zero section of its degree-zero
normal line bundle, hence its divisor bundle is unitary flat there.
Choose these neighborhoods disjoint when there is more than one
component, and shrink to a deck-invariant neighborhood. The pullback
of $\mathcal{O}_M(D)$ is the divisor bundle of the reduced inverse image of
$D$, so it admits a flat metric there. The geometric mean of its
finitely many deck translates is a deck-invariant flat metric and
descends to $\mathcal{O}_M(D)$ near $D$. Proposition~\ref{ued},
(3)$\Rightarrow$(4), gives the desired vertical linearization.

\subsection{Proof of Corollary \ref{stein}}
By the proof of Theorem \ref{thm:dps-semipositivity}, we obtain a pluriharmonic 
function $r$ defined on $X=S\setminus C$ away from a compact set, satisfying
$r+\log|w|^2$  is locally bounded near $C$, where
$w$ is a local defining function of $C$.
Then for $R_0 \gg 1$, the level set $\{r=R_0\}$ is a nonempty compact Levi-flat real hypersurface in $S\backslash C$, hence $S\backslash C$ cannot be Stein.

\bibliographystyle{alpha}
\bibliography{ref}

\end{document}